\documentclass[11pt]{article}

    \usepackage{url}
    \usepackage{verbatim}
    \usepackage[titletoc]{appendix}
    \usepackage{hyperref}
    \usepackage{bm}
	\usepackage{amsthm}
	\usepackage{amsmath}
	\usepackage{amsfonts}
    \usepackage{nicefrac}
    \usepackage{longtable}
    \usepackage{color}
    \usepackage{graphicx, amssymb,graphics}
    \usepackage{algorithmicx,algorithm}
    \usepackage[noend]{algpseudocode}

    \usepackage{paralist}

\newtheorem{rem}{Remark}[section]
\newtheorem{thm}{Theorem}[section]

\newcommand{\reff}[1]{{\rm (\ref{#1})}}
	\newcommand\be {\begin{equation}}
	\newcommand\ee {\end{equation}}

\newcommand{\ve}{\varepsilon}          

\numberwithin{equation}{section}

\graphicspath{{..}}

\title{ Local Structure-Preserving Relaxation Method for Charged Systems on Unstructured Meshes}
\date{\today}

\begin{document}
	
\author{
Zhonghua Qiao\thanks{Department of Applied Mathematics and Research Institute for Smart Energy, The Hong Kong Polytechnic University, Hung Hom, Hong Kong. (zhonghua.qiao@polyu.edu.hk)},~
\and
Zhenli Xu\thanks{School of Mathematical Sciences, MOE-LSC, and CMA-Shanghai and Shanghai Center for Applied Mathematics, Shanghai Jiao Tong University, Shanghai 200240, China. (xuzl@sjtu.edu.cn)},~
\and
Qian Yin\thanks{School of Mathematical Sciences, MOE-LSC, and CMA-Shanghai and Shanghai Center for Applied Mathematics, Shanghai Jiao Tong University, Shanghai 200240, China. (sjtu\_yinq@sjtu.edu.cn)},~
\and
Shenggao Zhou\thanks{School of Mathematical Sciences, MOE-LSC, and CMA-Shanghai and Shanghai Center for Applied Mathematics, Shanghai Jiao Tong University, Shanghai 200240, China. (sgzhou@sjtu.edu.cn)}
}

\maketitle

\begin{abstract}	
This work considers charged systems described by the modified Poisson--Nernst--Planck (PNP) equations, which incorporate ionic steric effects and the Born solvation energy for dielectric inhomogeneity. Solving the steady-state modified PNP equations poses numerical challenges due to the emergence of sharp boundary layers caused by small Debye lengths, particularly when local ionic concentrations reach saturation. To address this, we first reformulate the steady-state problem as a constraint optimization, where the ionic concentrations on unstructured Delaunay nodes are treated as fractional particles moving along edges between nodes. The electric fields are then updated to minimize the objective free energy while satisfying the discrete Gauss's law. We develop a local relaxation method on unstructured meshes that inherently respects the discrete Gauss's law, ensuring curl-free electric fields. Numerical analysis demonstrates that the optimal mass of the moving fractional particles guarantees the positivity of both ionic and solvent concentrations. Additionally, the free energy of the charged system consistently decreases during successive updates of ionic concentrations and electric fields. We conduct numerical tests to validate the expected numerical accuracy, positivity, free-energy dissipation, and robustness of our method in simulating charged systems with sharp boundary layers.

\bigskip


\noindent
\end{abstract}

\section{Introduction}
Charged systems play a crucial role in various practical applications, including biological ion channels~\cite{E:CP:98}, microfluidics~\cite{Schoch:RMP:08}, semiconductors~\cite{2012semiconductor}, and electrochemical devices~\cite{BTA:PRE:04}. The PNP theory, which describes the electro-diffusion of charged particles, is widely used in these contexts. However, the PNP theory assumes point-like charges and neglects the effects of particle size and Coulombic correlation, which limits its accuracy. To address these limitations, several modified theories have been proposed to incorporate these effects within the framework of the PNP theory~\cite{BAO:PRL:1997, HorngLinLiuBob_JPCB12, LJX:SIAP:2018, Wu:JCP:2002}. For example, the entropy of solvent molecules is considered to account for ionic size effects, using statistical mechanics applied to ions and solvent molecules on lattices~\cite{KBA:PRE:2007, BZLu_BiophyJ11}. Additionally, Coulombic ion correlations are incorporated by including self energies of ions~\cite{LJX:SIAP:2018, MXZ:SIAP:2021}.

Considerable efforts have been dedicated to the advancement of numerical techniques for simulating charged systems~\cite{Ding19JCP, HuHuang_NM20,LiuChun2020positivity,  LM2020, Metti:JCP:2016, Qian2021positive, SJXJ2021}. These numerical methods, ranging from finite difference to finite volume approaches, possess the ability to uphold various desirable properties at the discrete level, such as the positivity of numerical solutions, conservation of ionic mass, and dissipation of free energy. Such numerical advantages play a critical role in accurately simulating complex charged systems. Another desired feature is the capability to handle sharp boundary layers, in which ionic concentrations, electric fields, and dielectric coefficients may undergo significant variations. In such cases, stability becomes a major concern. The traditional central-differencing type of discretization may lead to severe spurious oscillations in sharp boundary layers due to the small Debye length. To address this issue, stabilized methods have been developed to mitigate spurious oscillations in cases dominated by convection~\cite{Bauer2012CMAME,Liu2020Entropy, Qiao2018IET, Stynes2005ActaN,Wang2002CMA}. For example, a stabilized finite element method based on the variational multiscale framework has been proposed to prevent oscillations in simulations of multi-ion transport in electrochemical systems~\cite{Bauer2012CMAME}. In the work by Qiao et al.~\cite{Qiao2018IET}, a Petrov-Galerkin least square method is introduced to simulate ion-flow environments, where a stabilization term based on the Péclet number and mesh size is incorporated into the Galerkin weak form~\cite{Qiao2018IET}. Additionally, the Scharfetter-Gummel fluxes, which automatically reduce to upwinding fluxes, have been widely employed in semiconductor simulations to handle large convection~\cite{Liu2020Entropy,Stynes2005ActaN, Wang2002CMA}.

In this study, we investigate modified PNP equations that incorporate ionic steric effects and the Born solvation energy arising from dielectric inhomogeneity. The inclusion of solvent concentration in the model to capture ionic steric effects poses additional challenges in the development of numerical schemes that preserve positivity. Moreover, the presence of Born solvation energy terms introduces additional convection in inhomogeneous dielectric environments. To handle the large convection, we have previously proposed structure-preserving numerical schemes utilizing Scharfetter--Gummel numerical fluxes for solving these modified PNP equations~\cite{Qiao2022NumANP}. While the positivity of ionic concentrations has been proven to be unconditionally preserved at the discrete level, the positivity of solvent concentration is not guaranteed, particularly when the ionic concentrations saturate locally in boundary layers.

We present a structure-preserving relaxation (SPR) method, for solving the steady-state behavior of complex charge systems characterized by sharp boundary layers resulting from the small Debye length. Instead of employing the traditional upwind-type discretization, our SPR method transforms the problem into a constrained optimization formulation, treating ionic concentrations as fractional particles that move along the edges of unstructured Delaunay meshes. To ensure the conservation of total ionic mass, the mass leaving one control volume enters adjacent volumes. The optimal mass of the moving fractional particle is determined in a manner that guarantees the positivity of both ionic concentrations and solvent concentration on Delaunay nodes. The electric fields are then updated to minimize the objective free energy while remaining on the constraint manifold. As a result, the system's free energy monotonically decreases with each update of ionic concentrations and electric fields. To achieve a curl-free electric field on unstructured meshes, we have developed a local relaxation method that inherently respects the discrete Gauss's law. This method extends the original approach, designed for rectangular finite-difference meshes, to unstructured meshes~\cite{BSD:PRE:2009,MR:PRL:2002}. This extension allows for flexible resolution of sharp boundary layers at irregular interfaces in complex charge systems and holds promise for applications in molecular simulations~\cite{PD:JPCM:2004} and plasma simulations~\cite{DIAMOND20211}. Numerical experiments are conducted to demonstrate the accuracy, structure-preserving properties, and robustness of the proposed method in simulations of charged systems with sharp boundary layers.

The rest of this paper is organized as follows. In Section 2, we present the modified Poisson--Nernst--Planck equations at the steady state and the equivalent constraint optimization formulation. In Section 3, we introduce the discretization mesh and relaxation algorithm. The corresponding numerical analysis on structure-preserving properties is given in Section 4. Numerical results are elaborated to demonstrate the performance of the proposed SPR algorithm in Section 5. Conclusions are given in Section 6.


\section{Modified Poisson--Nernst--Planck equations}
Consider ion dynamics in an electrolyte solution that includes mobile ions of $M$ species confined in
a domain $\Omega$.
The static electric field $\bm{E}$ satisfies $\nabla \times\bm{E}=\bm{0}$ and the Gauss's law
\begin{equation}
	\label{gausslaw}
	\nabla \cdot \varepsilon_0 \varepsilon_r \bm{E}=\rho,
\end{equation}
where $\varepsilon_0$ and $\varepsilon_r$ are the vacuum dielectric permittivity and dielectric coefficient of the medium, and $\rho$ is the charge density defined as
$$\rho=\sum_{s=1}^M z_s e c_s+fe.$$
Here $c_s$ is the ionic concentration for the $s$th species $(s=1,\cdots,M)$, $z_s$ is the ionic valence, $e$ is the elementary charge, and $f$ is the distribution of fixed charges. For the confined system, we assume periodic boundary conditions (or zero-flux boundary conditions) for ionic concentrations. The total ionic mass is given by
\begin{equation}\label{Mass}
\int_{\Omega} c_s \,d\bm{x} = N_s,~ s= 1, \cdots, M,
\end{equation}
where $N_s$ is a given total number of the $s$th species of ions. The charge neutrality is assumed to guarantee the existence of $\bm{E}$ with periodic boundary conditions: $$\sum_{s=1}^M z_s N_s + \int_{\Omega} f d\bm{x}=0.$$

We consider modified Nernst--Planck equations that are developed for description of ion dynamics~\cite{Qiao2022ANP}: for $~s =1, \cdots, M,$
\begin{equation}
	\label{MNP}
\frac{\partial c_s}{\partial t}=\nabla \cdot \gamma_s \left[ \nabla c_s-\frac {z_s e c_s}{k_B T} \bm{E} -\frac{a_s^3}{a_0^3}c_s \nabla \log\left(a_0^3 c_0 \right)+ \frac{\left(z_s e\right)^2c_s}{8\pi k_B T a_s} \nabla \left(\frac{1}{\varepsilon_0 \varepsilon_r}-\frac{1}{\varepsilon_0 }\right) \right], 
\end{equation}
where $\gamma_s$ is the diffusion coefficient, $k_BT$ is the thermal energy, $a_s$ represents the ionic size, $a_0$ represents the solvent molecule size, and $c_0$ is the concentration of solvent molecule defined by
$$
c_0=\frac{1-\sum_{v=1}^M a_v^3c_v}{ a_0^3}.
$$
The first two terms on the right-hand side of Eq.~\reff{MNP} represent the ionic electro-diffusion under the influence of an electric field. The third term takes into consideration the effect of ionic size, while the last term describes the Born solvation energy, accounting for dielectric inhomogeneity.

It has been demonstrated that the ion dynamics governed by Eq.~\reff{MNP} with periodic boundary conditions corresponds to an $H^{-1}$-gradient flow of a free-energy functional, as shown in a recent study~\cite{Qiao2022ANP}:

\begin{equation}\label{dF/dt}
\frac{dF}{dt} = -\sum_{s=1}^M \int_{\Omega} \frac{1}{c_s} \left|\nabla c_s-\frac {z_s e c_s}{k_B T} \bm{E} -\frac{a_s^3}{a_0^3}c_s \nabla \log\left(a_0^3 c_0 \right)+ \frac{\left(z_s e\right)^2c_s}{8\pi k_B T a_s} \nabla \left(\frac{1}{\varepsilon_0 \varepsilon_r}-\frac{1}{\varepsilon_0 } \right) \right|^2   d\bm{x} \leq 0,
\end{equation}
where
\[
F=\int_{\Omega} \frac{\varepsilon_0 \varepsilon_r}{2} \left|\bm{E} \right|^2+ k_B T\sum_{s=0}^M c_s\left( \log \left(a_s^3 c_s\right) -1 \right) + \sum_{s=1}^M c_s \frac{\left(z_s e\right)^2}{8\pi a_s} \left(\frac{1}{\varepsilon_0 \varepsilon_r}-\frac{1}{\varepsilon_0} \right)d\bm{x}.
\]
By the energy dissipation~\reff{dF/dt}, one can find that the steady-state ionic concentrations satisfy
\[
\nabla c_s-\frac {z_s e c_s}{k_B T} \bm{E} -\frac{a_s^3}{a_0^3}c_s \nabla \log\left(a_0^3 c_0 \right)+ \frac{\left(z_s e\right)^2c_s}{8\pi k_B T a_s} \nabla \left(\frac{1}{\varepsilon_0 \varepsilon_r}-\frac{1}{\varepsilon_0 } \right) = \bm{0}.
\]

For ease of presentation, we derive dimensionless formulation via rescaling by characteristic concentration $c_B$, diffusion constant $\gamma_0$, length $L$, and Debye length $\lambda_D=\sqrt{\varepsilon_0 k_BT/\left(2e^2c_B\right)}$. Define $\tilde{\bm{x}} = \bm{x}/L$, $\tilde{t} = t\gamma_0 / ( L\lambda_D)$, $\tilde{\nabla}= L \nabla$, $\tilde{f}= f/(ec_B)$, $\tilde{\rho}= \rho/(ec_B)$, $\tilde{\gamma}_s = \gamma_s/ \gamma_0$, $\tilde{a}_s = a_s/L$, $\tilde{c}_s= c_s/c_B$, $\tilde{\bm{E}}=Le \bm{E} /(k_B T )$, and $\tilde{N}^{\ell} = N^{\ell}/(L^3 c_B)$. We drop the tildes over all new variables and obtain the steady state of the system governed by
\begin{equation}
	\label{ss-manp}
	\left\{
	\begin{aligned}
		&  \nabla c_s-z_sc_s\bm{E}-\frac{a_s^3}{a_0^3}c_s \nabla \log \left(a_0^3 c_0\right)+\chi \frac{z_s^2 c_s}{a_s} \nabla \left(\frac{1}{\varepsilon_r}-1\right) =\bm{0}, \\
				&\int_{\Omega} c_s \,d\bm{x} = N_s,~  s =1, \cdots, M,  	\\
		&\nabla \cdot 2\kappa^2 \varepsilon_r \bm{E}=\sum_{s=1}^M z_s c_s+f,\\
		& \nabla \times\bm{E}=\bm{0},\\
	\end{aligned}
	\right.
\end{equation}
where $\kappa=\lambda_D/L$ and $\chi = e^2/(k_B T 8\pi L \ve_0) $ are two dimensionless parameters related to length scales. Such a system corresponds to a unique minimizer of the constrained optimization problem:
\begin{equation}\label{ConOpt}
\begin{aligned}
\underset{c_1,\cdots,c_M,\bm{E}}{\min} \mathcal{F}(c_1,\cdots,c_M,\bm{E})&:=\int_\Omega \kappa^2 \varepsilon_r \left|\bm{E} \right|^2+ \sum_{s=0}^M \left(  c_s \log c_s  + \chi c_s \frac{z_s ^2}{ a_s} \frac{1}{\varepsilon_r}
	\right) d\bm{x}, \\
 \mbox{s.t. } K(c_1,\cdots,c_M,\bm{E})&:=\nabla \cdot 2\kappa^2 \varepsilon_r \bm{E}-\sum_{s=1}^M z_s c_s-f =0, \\
H_s(c_s)&:=\int_\Omega c_s d\bm{x}-N_s = 0.
\end{aligned}	
\end{equation}
%
For ease of presentation, we denote by $X=(X_{c_1}, \cdots, X_{c_M}, X_{\bm{E}})=(c_1,\cdots,c_M, \bm{E})$ and introduce perturbations $Y^{i}=(Y^{i}_{c_1}, \cdots, Y^{i}_{c_M} ,  Y^{i}_{\bm{E}})$ ($i=1, 2$) that satisfy the constraints
\[
\nabla \cdot 2\kappa^2 \varepsilon_r Y^{i}_{\bm{E}}-\sum_{s=1}^M z_s Y^i_{c_s} =0 \mbox{~and~} \int_\Omega Y^{i}_{c_s} d\bm{x}= 0.
\]
The first variation of $\mathcal{F}$ reads
\[
	\begin{aligned}
		\delta \mathcal{F} [X][Y^1]&:=\lim_{t\to 0}\frac{ \mathcal{F}[X+tY^1]-\mathcal{F}[X]}{t}  \\
		&= \int_\Omega 2 \kappa^2 \varepsilon_r Y^1_{\bm{E}} \cdot X_{\bm{E}} + \sum_{s=1}^M Y^1_{c_s} \left[ \log(a_s^3X_{c_s}) - \frac{a_3^3}{a_0^3} \log(a_0^3c_0) + \chi \frac{z_s ^2}{ a_s \ve_r} \right]  d\bm{x}.
	\end{aligned}
\]
Furthermore, the second variation of $\mathcal{F}$ reads
\[
\begin{aligned}
\delta^2 \mathcal{F} [X][Y^2, Y^1]&= \lim_{t\to 0}\frac{\delta \mathcal{F}[X+tY^2][Y^1]-\delta \mathcal{F}[X][Y^1]}{t} \\
&=\int_\Omega 2\kappa^2 \varepsilon_rY^1_{\bm{E}} \cdot Y^2_{\bm{E}}+\sum_{s=1}^M \frac{Y_{c_s}^1 Y^2_{c_s}}{X_{c_s}}+\sum_{s=1}^M \frac{a_s^6}{a_0^6}\frac{Y_{c_s}^1 Y^2_{c_s}}{c_0} d\bm{x}.
\end{aligned}
\]

It can be shown that $\delta^2 \mathcal{F} [X][Y^1, Y^1] > 0$, indicating that the functional $\mathcal{F}$ is strictly convex under the given constraints. The existence of a solution to the constrained optimization problem~\reff{ConOpt} can be proven using the direct method in the calculus of variations~\cite{Li:N:2009, Li:SJMA:2009}, while the uniqueness of the solution can be established by exploiting the strict convexity of $\mathcal{F}$. The unique minimizer is characterized by the Karush--Kuhn--Tucker (KKT) conditions~\cite{NocedalWright_Book1999}, which can be derived by constructing the Lagrangian function:

\begin{equation}
		\mathcal{L}:=\int_\Omega \left[ \kappa^2 \varepsilon_r \left|\bm{E} \right|^2+ \sum_{s=0}^M \left(  c_s \log c_s  + \chi c_s \frac{z_s ^2}{ a_s} \frac{1}{\varepsilon_r}
	\right)
		-\lambda_\phi K(c_1,\cdots,c_M,\bm{E})  \right] d\bm{x} -   \sum_{s=1}^M \lambda_s H_s(c_s). \label{func}
	\end{equation}
The unique minimizer is obtained by taking variation of $\mathcal{L}$ with respect to $\bm{E}$, $c_s$, $\lambda_\phi$, and $\lambda_s$ as follows:
	\begin{align}
		\label{lE}
		\frac{\delta \mathcal{L}}{\delta \bm{E}}=0 &\Rightarrow
		\bm{E}=-\nabla \lambda_\phi,\\
		\label{chemicalp}
		\frac{\delta \mathcal{L}}{\delta c_s}=0 &\Rightarrow
		\lambda_s=\log \left(a_s^3 c_s\right)-\frac{a_s^3}{a_0^3} \log \left(a_0^3 c_0\right)+2+\chi \frac{z_s^2}{a_s} \frac{1}{\ve_r}+z_s\lambda_\phi,  \\
		\label{lgauss}
		\frac{\delta \mathcal{L}}{\delta \lambda_\phi}=0 &\Rightarrow
		 \nabla\cdot 2\kappa^2\varepsilon_r \bm{E}=\sum_{s=1}^M z_s c_s+f,\\
		\label{lcons}
		\frac{\delta \mathcal{L}}{\delta \lambda_s}=0 &\Rightarrow
		\int_\Omega c_s d\bm{x}=N_s.
	\end{align}
	Taking curl of Eq.~\eqref{lE}, one recovers
	\begin{equation}\label{curlE}
		\nabla\times \bm{E}=0.
	\end{equation}
	Taking gradient of ~\eqref{chemicalp} and using~\eqref{lgauss}, one recovers the whole steady-state system~\eqref{ss-manp}.

Based on the constraint optimization problem~\reff{ConOpt}, we propose an extension of the curl-free algorithm~\cite{BSD:PRE:2009,MR:PRL:2002}, which was initially designed for particle simulations on rectangular finite-difference meshes, to unstructured meshes. Additionally, we develop a structure-preserving local relaxation method to solve the system~\reff{ss-manp}. In contrast to conventional optimization methods, our local relaxation method is specifically tailored for unstructured meshes, ensuring that the concentrations and electric field are strictly confined to the constraint manifold, {where mass and Gauss's law are conserved}.


\section{Relaxation Algorithm}

\subsection{Mesh and Discretization}

Let $\Omega$ be a bounded convex polygonal domain covered by a primal Delaunay triangulation and dual Voronoi polygons. To simplify the presentation, we will focus on the 2D case, but the extension to 3D is straightforward. In this work, we adopt the notations, algorithm, and analysis results used in the mimetic finite difference (MFD) method \cite{Adler2021AFF,Rodrigo2015FEMFD,Vabishchevich2005MFD}.

\begin{figure}[h!]
	\centering
	\includegraphics[scale=0.65]{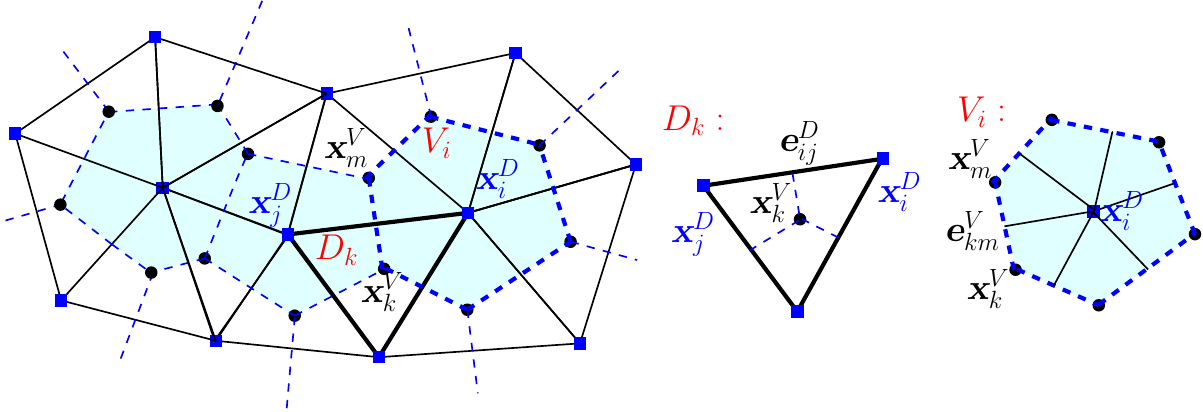}
	\caption{Delaunay triangulation $D_k$ with blue square vertices ($\bm{x}_i^{D}$, $\bm{x}_j^{D}$, $\cdots$) and dual Voronoi control volumes $V_i$ with black round vertices ($\bm{x}^V_k$, $\bm{x}^V_m$, $\cdots$). Each Voronoi vertex $\bm{x}^V_k$ is the circumcenter of the Delaunay element $D_k$. Vectors $\bm{e}_{i j}^{D}$ and $\bm{e}_{km}^{V}$ are unit vectors. }
	\label{f:mesh}
\end{figure}
As shown in Figure~\ref{f:mesh}, Delaunay nodes (D-nodes) are denoted as $\bm{x}^D_i$ for $i=1, 2,\cdots, M_D$, with $M_D$ being the total number of Delaunay nodes. Each Delaunay node $\bm{x}^D_i$ is associated with a Voronoi polygon
$$
V_{i}=\left\{\bm{x}|\bm{x} \in \Omega,| \bm{x}-\bm{x}_{i}^{D}|<| \bm{x}-\bm{x}_{j}^{D} \mid, j=1,2, \ldots, M_{D}\right\}, \quad i=1,2, \ldots, M_{D},
$$
and its boundary is denoted by $\partial V_i$.
The Voronoi vertices (V-nodes) are denoted by $\bm{x}^V_k$ for $k=1,2,\cdots M_V$, where $M_V$ is the total number of Voronoi nodes. Each Voronoi vertex, e.g., $\bm{x}^V_k$, is the circumcenter of the Delaunay element $D_k$, whose boundary is denoted by $\partial D_k$.
The common edge of the Delaunay elements $D_k$ and $D_m$ is denoted by $\partial D_{km}$, i.e.
$
\partial D_{k m}=\partial D_{k} \cap \partial D_{m}, \quad k \neq m$, $k, m=1,2, \ldots, M_{V} .
$
Analogously, the common edge of the Voronoi polygons $V_i$ and $V_j$ is defined by $\partial V_{ij}$. For each $D_{k}$, $\mathcal{W}^{D}(k)$ denotes the indices of neighboring Delaunay elements that share common edges with $D_{k}$, i.e.
$$
\mathcal{W}^{D}(k)=\left\{m \mid \partial D_{k} \cap \partial D_{m} \neq \emptyset, m=1, \ldots, M_{V}\right\}, \quad k=1,2, \ldots, M_{V}.
$$
Analogously, for each Voronoi polygon $V_{i}$, $\mathcal{W}^{V}(i)$ denotes the indices of neighboring Voronoi elements that share common edges with $V_{i}$.

The distances between connected D-nodes and V-nodes are defined by
$$
\begin{aligned}
	d_{ij}^D&=|\bm{x}_i^D-\bm{x}_j^D|,\quad i=1,2,\cdots ,M_D, j\in\mathcal{W}^V(i),\\
	d_{km}^V&=|\bm{x}_k^V-\bm{x}_m^V|,\quad k=1,2,\cdots ,M_V, m\in\mathcal{W}^D(k).
\end{aligned}
$$
The following measures are introduced
$$
\begin{aligned}
&m\left(D_{k}\right)=\int_{D_{k}} d \bm{x}, \quad m\left(\partial D_{k m}\right)=\int_{\partial D_{k m}} d \bm{x}, \quad k=1,2, \ldots, M_{V},  ~m\in \mathcal{W}^{D}(k),\\
&m\left(V_i\right)=\int_{V_i} d \bm{x}, \quad m\left(\partial V_{ij}\right)=\int_{\partial V_{ij}} d \bm{x}, \quad i=1,2, \ldots, M_{D}, ~j\in \mathcal{W}^{V}(i).
\end{aligned}
$$
Scalar functions are approximated by piecewise-constant grid functions on the D-nodes. Let the function space
$$
H_{D}=\left\{u(\bm{x}) \mid u(\bm{x})=u_{i}^{D}:=u\left(\bm{x}_{i}^{D}\right), ~\forall \bm{x} \in V_i, ~ i=1,2, \ldots, M_{D}\right\},
$$
which is equipped with a scalar inner product and norm
\begin{equation}\label{scaprod}
(u,v)_{D}=\sum_{i=1}^{M_D} u_i^D v_i^D m(V_i), \quad \|u\|_{D}=(u,u)_{D}^{1/2}.
\end{equation}
Unit vectors along edges pointing from a D-node with a smaller index to a neighboring one with a larger index are
$$
\bm{e}_{i j}^{D}=\bm{e}_{j i}^{D}, \quad i=1,2, \ldots, M_{D}, \quad j \in \mathcal{W}^{V}(i).
$$
Vector functions $\bm{u}(\bm{x})$ are approximated along Delaunay edges by the projected values on midpoints of the edges:
$$
u_{i j}^{D}=\bm{u}(\bm{x}_{i j}^{D}) \cdot \bm{e}_{i j}^{D}, \quad i=1,2, \ldots, M_{D},  ~j \in \mathcal{W}^{V}(i),\quad
\bm{x}_{i j}^{D}=\frac{1}{2}\left(\bm{x}_{i}^{D}+\bm{x}_{j}^{D}\right).
$$
We denote by $\bm{H}_D$ the set of grid vector functions determined by $u_{i j}^{D}$ for $i, j =1,2, \ldots, M_{D}$.
The scalar product and the norm in $\bm{H}_{D}$ are defined by
\begin{equation}\label{vecprod}
\begin{aligned}
	(\bm{u}, \bm{v})_{D} &= \sum_{k=1}^{M_{V}} \sum_{m \in \mathcal{W}^{D}(k)} \sum_{(i, j) \in \mathcal{Q}^{D}(k, m)} u_{i j}^{D} v_{i j}^{D}\left|\bm{x}_{m}^{V}-\bm{x}_{k m}^{V}\right| m\left(\partial D_{k m}\right), \\
	\|\bm{u}\|_{D} &=(\bm{u}, \bm{u})_{D}^{1/2},
\end{aligned}
\end{equation}
where
$$
\mathcal{Q}^{D}(k, m)=\left\{(i, j) \mid \bm{x}_{i}^{D}, \bm{x}_{j}^{D} \in \partial D_{k m}, i=1,2, \ldots, M_{D}, j \in \mathcal{W}^{V}(i)\right\}
$$
is the set of vertices of the plane $\partial D_{k m}$, and $\bm{x}_{km}^V$ is a point defined as the intersection of the Voronoi edge $\bm{x}_{m}^{V}$ $\bm{x}_{k}^{V}$ and the plane $\partial D_{km}$. In 2D, $\bm{x}_{km}^V$ is coincided with $\bm{x}_{ij}^{D}$.

Denote by $c_{s}^h(\bm{x})\in H_D,s=0,1,\cdots M,$ the approximation of ionic concentrations $c_{s}(\bm{x})$, with $c_{s}^h(\bm{x})=c_{s,i}^D:=c_{s}^h(\bm{x}_i^D)$, $\forall \bm{x} \in V_i$, $i=1,2,\cdots,M_D$. Analogously, one defines a piecewise-constant function $f^h \in H_D$ for the distribution of fixed charges. Let $\bm{E}^h(\bm{x})\in \bm{H}_D$ be the approximation of  electric fields,  with $E^D_{ij}:=\bm{E}^h(\bm{x}_{ij}^D)\cdot \bm{e}_{i j}^{D}$.
The ionic mass conservation~\eqref{Mass} is discretized by
\begin{equation}\label{dismasscons}
	\sum_{i=1}^{M_D} c_{s,i}^{D} m(V_i)=N_s,\quad \forall \text{ } s=1,\cdots,M.
\end{equation}
For the scalar dielectric coefficient $\varepsilon_r^h(\bm{x})\in H_D$, an average operator $\mathcal{A}$ is introduced which is defined by
$$
\mathcal{A}(\varepsilon_r^h)\left(\bm{x}_{ij}^D\right)=\varepsilon_{ij}^D:=\frac{\varepsilon_{i}^D+\varepsilon_{j}^D}{2},~ i=1,\cdots M_D, ~j\in\mathcal{W}^{V}(i).
$$
In addition, one introduces a sign function
$$
\tau_{\{i<j\}}=\begin{cases}
	1,&\quad i<j,\\
	-1,& \quad i>j.
\end{cases}$$
The Gauss's law~\eqref{gausslaw} is discretized as
\begin{equation}\label{disgausslaw}
	 2\kappa^2 \left(\operatorname{div}_D \mathcal{A}(\varepsilon^h) \bm{E}^h\right)_i^D=\sum_{s=1}^M  z_s c_{s,i}^D +f^{D}_{i},\quad \forall \text{ } i=1,\cdots,M_D,
\end{equation}
where the discrete divergence operator $\operatorname{div}_D:\bm{H}_D \to H_D$ is defined by
\begin{equation}\label{disdiv}
\left(\operatorname{div}_D \bm{y}\right)_i^D=\frac{1}{m( V_i)} \sum_{j\in \mathcal{W}^V(i)} \left(\bm{n}^V_{ij}\cdot \bm{e}^D_{ij} \right) y_{ij}^D  m(\partial V_{ij}) = \frac{1}{m( V_i)} \sum_{j\in \mathcal{W}^V(i)} \tau_{\{i<j\}}  y_{ij}^D  m(\partial V_{ij}).
\end{equation}
Here, $\bm{y}\in \bm{H}_D$ and $\bm{n}^V_{ij}$ is the normal vector to the edge $\partial V_{ij}$ with respect to $V_i$.


By the definition of discrete inner products, the energy functional defined in~\eqref{ConOpt} can be discretized by
\begin{equation}\label{disf}
	\mathcal{F}_h= \kappa^2 \left(\mathcal{A}(\varepsilon^h)\bm{E}^h,\bm{E}^h\right)_{D}+\sum_{s=0}^M \left(c_s^h,\log c_s^h\right)_D
	+\chi\sum_{s=1}^M \frac{z_s^2}{a_s} \left(c_s^h,\frac{1}{\varepsilon^h}\right)_D.
\end{equation}
In order to solve the constrained optimization problem~\reff{ConOpt}, we employ a local update scheme for the ionic concentrations and electric fields. The ionic concentrations, represented as fractional particles on nodes, are moved between adjacent Voronoi control volumes in order to minimize the discrete energy function~\eqref{disf}. Simultaneously, the electric fields on the edges connecting these control volumes are updated to ensure that the discrete Gauss's law is preserved.


\subsection{Updates of Concentrations and Electric Fields}\label{ss:CEUpdate}
To solve the system \reff{ss-manp}, we employ a constrained optimization approach that minimizes the discrete energy function \eqref{disf} while respecting the constraints specified in \reff{ConOpt}. Assuming that the initial data of the ionic concentrations and electric fields satisfy these constraints, we treat the ionic concentrations as fractional particles and propose a local update scheme. This scheme involves moving fractional particles from one Voronoi control volume to an adjacent one, with the objective of minimizing the discrete energy function \eqref{disf}. Simultaneously, we update the electric fields on the edges connecting adjacent control volumes to ensure that the discrete Gauss's law is maintained.

  \begin{figure}[ht]
  	\setlength{\abovecaptionskip}{-0.cm}
  	\setlength{\belowcaptionskip}{-0.cm}
  	\centering
  	\includegraphics[scale=0.65]{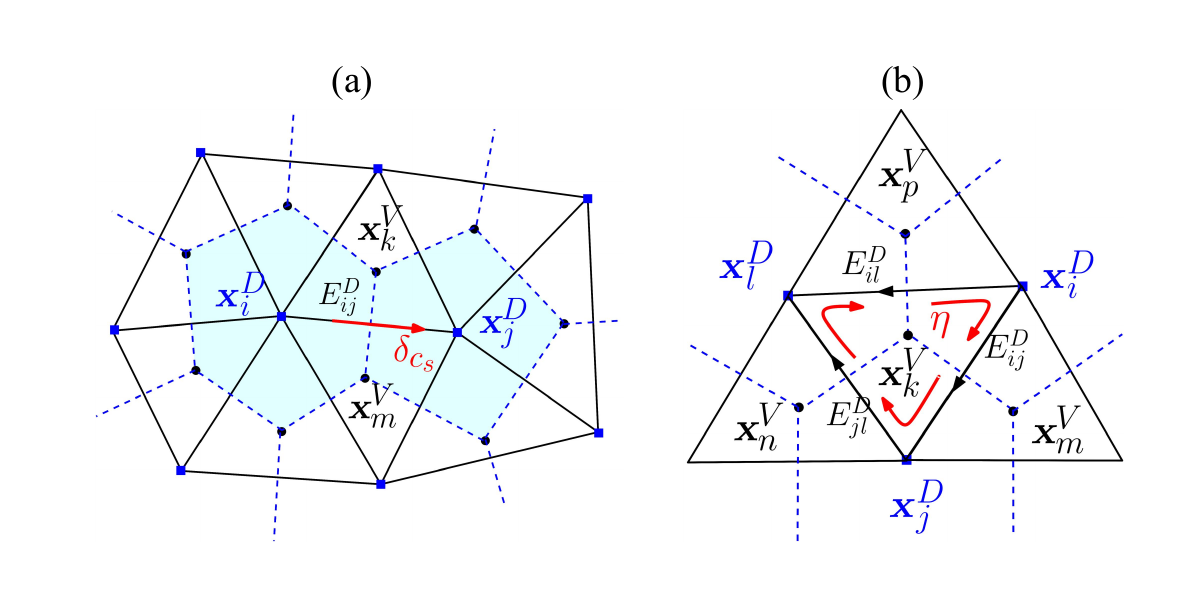}
  	\caption{(a): The diagram for updates of concentrations and corresponding electric fields. (b): The diagram for update of electric fields in a Delaunay element.}
  	\label{f:updatec}
  \end{figure}


Consider two adjacent D-nodes $\bm{x}^D_i$ and $\bm{x}^D_j$, which are associated with Voronoi polygons $V_i$ and $V_j$, respectively. Refer to the schematic diagram shown in Figure \ref{f:updatec} (a). Let the ionic concentrations of the $s$th species on the D-nodes be denoted as $c^{D}_{s,i}$ and $c^{D}_{s,j}$. The electric field along the edge $\bm{e}_{i j}^{D}$ is represented by $E_{ij}^D$. Without loss of generality, assume that $i<j$. Suppose that a certain number of fractional particles, with mass $\delta_{ c_s}$, for the $s$th ionic species move from $V_i$ to $V_j$.
Thus, $c^{D}_{s,i}$ and $c^{D}_{s,j}$ are updated by
 \begin{equation}\label{deltac}
 	\tilde{c}^{D}_{s,i} =c^{D}_{s,i}-\frac{\delta_{c_s}}{m(V_i)} \text { and } \tilde{c}^{D}_{s,j} =c^{D}_{s,j}+\frac{\delta_{ c_s}}{m(V_j)}.
 \end{equation}
Correspondingly, the solvent concentrations $c^{D}_{0,i}$ and $c^{D}_{0,j}$ are updated by
\begin{equation}\label{deltac0}
	\tilde{c}^{D}_{0,i} =c^{D}_{0,i}+\frac{a_s^3}{a_0^3}\frac{\delta_{ c_s}}{m(V_i)} \text { and } \tilde{c}^{D}_{0,j} =c^{D}_{0,j}-\frac{a_s^3}{a_0^3}\frac{\delta_{ c_s}}{m(V_j)}.
\end{equation}
Here the tilde signs indicate the variables after update.
By the discrete Gauss's law~\eqref{disgausslaw}, the electric field across edge $\partial V_{ij}$ is updated by
	\begin{equation}\label{checkE}
		\check{E}_{ij}^D =E_{ij}^D+\delta_E,
	\end{equation}
	where
   \begin{equation}
	\delta_E =-\frac{z_s \delta_{ c_s}}{2\kappa^2\varepsilon_{ij}^D m(\partial V_{ij})}.
   \end{equation}
   Here the $\check{E}_{ij}^D$ is an interim step for the electric fields, and will be further updated in the next step, detailed in Section~\ref{ss:CurlFree}.
	As a result, the associated change in the energy~\eqref{disf} reads
	\begin{equation}\label{disfunc1}
	\begin{aligned}
		\Delta \mathcal{F}_1(\delta_{ c_s})=& \kappa^2 \varepsilon_{ij}^D d_{km}^V m(\partial D_{km})\left(2E_{ij}^D \delta_E +\delta_E^2\right)-\delta_{ c_s} \log  \frac{c_{s,i}^D-\delta_{ c_s}/m(V_i)}{c_{s,j}^D+\delta_{ c_s}/m(V_j)} \\
		&+ m(V_i) c_{s,i}^D \log \left[ 1-\frac{\delta_{ c_s}}{c_{s,i}^Dm(V_i)} \right]
		+ m(V_j) c_{s,j}^D \log \left[ 1+\frac{\delta_{ c_s}}{c_{s,j}^Dm(V_j)} \right] \\
		&+\frac{a_s^3\delta_{ c_s}}{a_0^3}\log \frac{1-\sum_{v=1}^M a_v^3 c_{v,i}^D+a_s^3 \delta_{ c_s}/m(V_i)}{1-\sum_{v=1}^M a_v^3 c_{v,j}^D-a_s^3 \delta_{ c_s}/m(V_j)}\\
		&+m(V_i)\frac{1-\sum_{v=1}^M a_v^3 c_{v,i}^D}{a_0^3}\log \left[1+\frac{a_s^3 \delta_{ c_s}}{m(V_i)\left(1-\sum_{v=1}^M a_v^3 c_{v,i}^D\right)}\right]\\
		&+m(V_j) \frac{1-\sum_{v=1}^M a_v^3 c_{v,j}^D}{a_0^3}\log \left[1-\frac{a_s^3 \delta_{ c_s}}{m(V_j)\left(1-\sum_{v=1}^M a_v^3 c_{v,j}^D\right)}\right]\\
		&-\frac{\chi z_s^2}{a_s}\delta_{ c_s} \left(\frac{1}{\varepsilon_i^D}-\frac{1}{\varepsilon_j^D}\right).
	\end{aligned}
    \end{equation}
$\Delta \mathcal{F}_1$ is a convex function of $\delta_{ c_s}$ as is shown in Theorem~\ref{thm-positive}. Thus, the optimal $\delta_{ c_s}^*$ is determined by minimizing the convex energy: $d \Delta \mathcal{F}_1/ d \delta_{ c_s}=0$, i.e.,
	\begin{equation}
		\begin{aligned}
			&- d_{km}^V m(\partial D_{km})\frac{z_s}{ m(\partial V_{ij})}\left(E_{ij}^D -\frac{z_s \delta_{ c_s}^*}{2\kappa^2\varepsilon_{ij}^D m(\partial V_{ij})} \right) -\log  \frac{c_{s,i}^D-\delta_{ c_s}^*/m(V_i)}{c_{s,j}^D+\delta_{ c_s}^*/m(V_j)}\\
			&+\frac{a_s^3}{a_0^3}\log \frac{1-\sum_{v=1}^M a_v^3 c_{v,i}^D+a_s^3 \delta_{ c_s}^*/m(V_i)}{1-\sum_{v=1}^M a_v^3 c_{v,j}^D-a_s^3 \delta_{ c_s}^*/m(V_j)}-\frac{\chi z_s^2}{a_s} \left(\frac{1}{\varepsilon_i^D}-\frac{1}{\varepsilon_j^D}\right)=0.
		\end{aligned}
	\label{equF'}
	\end{equation}

Theorem~\ref{thm-positive} provides a proof that the optimal mass $\delta_{ c_s}^*$, which satisfies equation~\reff{equF'}, ensures the numerical positivity of the updated ionic concentrations, namely $c^{D}_{s,i}-\delta_{ c_s}^*/m(V_i)$ and $c^{D}_{s,j}+\delta_{ c_s}^*/m(V_j)$, as well as the solvent concentrations $(1-\sum_{v=1}^M a_v^3 c_{v,i}^D+a_s^3 \delta_{ c_s}^*/m(V_i))/a_0^3$ and $(1-\sum_{v=1}^M a_v^3 c_{v,j}^D-a_s^3 \delta_{ c_s}^*/m(V_j))/a_0^3$. In numerical computations, the nonlinear scalar equation~\reff{equF'} is efficiently solved using the Newton's iteration method, incorporating a thresholding procedure to ensure numerical positivity.


\subsection{Local Curl-free Algorithm}\label{ss:CurlFree}
Minimizing the discrete energy function~\eqref{disf} under the constraint of the electric field, as indicated in~\reff{lE}, results in the curl-free condition~\reff{curlE}. The discrete Gauss's law remains unchanged after the previous update on ionic concentrations and electric fields. Our next step is to further minimize the discrete energy function with respect to the electric field, while ensuring that the discrete Gauss's law is still satisfied.


Consider a Delaunay element $D_k$ with nodes $\bm{x}_i^D$, $\bm{x}_{j}^D$, and $\bm{x}_{l}^D$, as shown in Figure~\ref{f:updatec} (b). The interim electric fields along the edges are denoted by $\check{E}_{ij}^D$, $\check{E}_{jl}^D$, and $\check{E}_{il}^D$, which are updated according to equation~\reff{checkE}. Let $D_m$, $D_n$, and $D_p$ be the adjacent Delaunay elements with circumcenters $\bm{x}^V_m$, $\bm{x}^V_n$, and $\bm{x}^V_p$, respectively. Without loss of generality, we assume $i<j<l$. As depicted in Figure~\ref{f:updatec} (b), the black arrows along the edges of $D_k$ indicate the predefined directions of the unit vectors. To maintain the discrete Gauss's law at each D-node, we introduce a uniform electric flux $\eta$ (positive or negative) rotating clockwise along the edges of each Delaunay element.
Specifically, the field updates are performed as

		\begin{equation}
			\begin{aligned}
				&\tilde{E}_{ij}^D=\check{E}_{ij}^D+\frac{\eta}{\varepsilon_{ij}^D m(\partial V_{ij})},\\
				&\tilde{E}_{jl}^D=\check{E}_{jl}^D+\frac{\eta}{\varepsilon_{jl}^D m(\partial V_{jl})},\\
				&\tilde{E}_{il}^D=\check{E}_{il}^D-\frac{\eta}{\varepsilon_{il}^D m(\partial V_{il})}.
			\end{aligned}
		\end{equation}
The induced change in energy function~\eqref{disf} reads
\begin{equation}\label{chgdisf}
		\Delta \mathcal{F}_2(\eta)={\kappa^2} \left[2\left(A_1 \check{E}_{ij}^D +A_2 \check{E}_{jl}^D -A_3 \check{E}_{il}^D \right)\eta+ \left( \frac{A_1}{\varepsilon_{ij}^D m(\partial V_{ij})} + \frac{A_2}{\varepsilon_{jl}^D m(\partial V_{jl})} + \frac{A_3}{\varepsilon_{il}^D m(\partial V_{il})} \right) \eta^2 \right],
\end{equation}
	where
	$$
	\begin{aligned}
		A_1&=d_{km}^V m(\partial D_{km})\frac{1}{ m(\partial V_{ij})},\\
		A_2&=d_{kn}^V m(\partial D_{kn})\frac{1}{ m(\partial V_{jl})},\\
		A_3&=d_{kp}^V m(\partial D_{kp})\frac{1}{ m(\partial V_{il})},
	\end{aligned}
	$$
	are mesh dependent variables.
Minimization of quadratic the function~\eqref{chgdisf} gives an explicit expression of the minimizer
	\begin{equation}\label{eta}
		\eta=-\frac{A_1 \check{E}_{ij}^D+A_2 \check{E}_{jl}^D-A_3 \check{E}_{il}^D}{A_1/(\varepsilon_{ij}^D m(\partial V_{ij}))+A_2/(\varepsilon_{jl}^D m(\partial V_{jl}))+A_3/(\varepsilon_{il}^D m(\partial V_{il}))}.
	\end{equation}
\begin{rem}
		 Let us follow~\cite{Vabishchevich2005MFD} to define
		$$\operatorname{curl}_D: \bm{H}_D \to \bm{H}_V, \quad \left(\operatorname{curl}_D \bm{E}_{ij}^D\right)_{km}^V=\tau_{\{k<m\}}\frac{1}{m(\partial D_{km})}\sum_{(i,j)\in \mathcal{Q}^D_{km}}\chi\left(\bm{n}_{km}^D,\bm{e}^D_{ij}\right)E_{ij}^D d_{ij}^D.$$
		In 2D, one haves $A_1=d_{ij}^D$, $A_2=d_{jl}^D$, $A_3=d_{il}^D$, and $\left(\operatorname{curl}_D \bm{E}_{ij}^D\right)_{km}^V=\pm(E_{ij}^D d_{ij}^D+E_{jl}^D d_{jl}^D-E_{il}^D d_{il}^D)/m(\partial D_{km})$ after simplification. By~\reff{eta}, one can observe that the electric field is guaranteed to be curl-free when the algorithm converges to the minimizer such that $\eta=0$.
\end{rem}
As long as the initial ionic concentration and electric fields satisfy the constraints, the abovementioned update schemes for ionic concentrations and electric fields respect them throughout the whole relaxation process. Algorithm~\ref{alg} combines the two update schemes, described in Sections~\ref{ss:CEUpdate} and~\ref{ss:CurlFree}, into a comprehensive algorithm for solving the constraint optimization problem~\reff{ConOpt}.

\begin{algorithm}[H]
	\label{alg}
	\caption{Structure-preserving relaxation (SPR) algorithm}
	Input: initial ionic concentrations, the electric field, and stopping criterion $\epsilon_{\rm tol}$
	\begin{algorithmic}[1]
		\While{Energy change $\Delta \mathcal{F}:=\sum \Delta \mathcal{F}_1+\sum \Delta \mathcal{F}_2 >$ $\epsilon_{\rm tol}$} 
		\State {\bf Step 1:}
		\For {each edge and each species of ionic concentrations}		
		\State Update concentrations ($\tilde{c}^{D}_{s,i} \leftarrow c^{D}_{s,i}$) and corresponding electric fields on the edge ($\check{E}_{ij}^D \leftarrow E_{ij}^D$) ; cf. Section~\ref{ss:CEUpdate}
		\State Compute associated energy change $\Delta \mathcal{F}_1$
		\EndFor
		\State {\bf Step 2:}
		\For {each Delaunay element}
		\State Perform local curl-free updates on electric fields ($\tilde{E}_{ij}^D \leftarrow  \check{E}_{ij}^D$); cf. Section~\ref{ss:CurlFree}
		\State Compute associated energy change $\Delta \mathcal{F}_2$
		\EndFor
		\EndWhile
	\end{algorithmic}
\end{algorithm}

\begin{rem}
		The proposed algorithm can be extended to 3D cases in a straightforward manner.
\end{rem}

\section{Numerical Analysis}
The subsequent theorems aim to analyze the numerical properties of the proposed algorithm.
\begin{thm}\label{thm-masscons}
	The proposed algorithm preserves the discrete total mass~\eqref{dismasscons} and discrete Gauss's law~\eqref{disgausslaw}.
\end{thm}
\begin{proof}
	 One first goes through the Step 1. After the update of ionic concentrations, the ionic total mass of each species is given by
	\[
	\sum_{l=1}^{M_D} \tilde{c}_{s,l}^D m(V_l)=\sum_{l=1,l\neq i, j}^{M_D} c_{s,l}^D m(V_l)+\left(c_{s,i}^D-\frac{\delta_{ c_s}^*}{m(V_i)} \right)m(V_i)+\left(c_{s,j}^D+\frac{\delta_{ c_s}^*}{m(V_j)} \right)m(V_j)=\sum_{l=1}^{M_D} c_{s,l}^D m(V_l)=N_s.
	\]
	Thus, the discrete total mass is conserved in the relaxation algorithm. One next checks the discrete Gauss's law~\eqref{disgausslaw} at each D-node. After Step 1 in the algorithm, the discrete Gauss's law at the D-node $\bm{x}_i^D$ becomes
	\[
	\begin{aligned}
	2\kappa^2 \left(\operatorname{div}_D \mathcal{A}(\varepsilon^h) \check{\bm{E}}^h\right)_i^D &=
		  2\kappa^2\left(\operatorname{div}_D \mathcal{A}(\varepsilon^h) \bm{E}^h\right)_i^D+ 2\kappa^2\frac{1}{m(V_i)} \tau_{\{i<j\}}\varepsilon_{ij}^D \delta_{E} m\left(\partial V_{ij}\right) \\
		 &=
		 \sum_{s=1}^M  z_s c_{s,i}^D +f^{D}_{i} +  \frac{-z_s\delta_{ c_s}^* }{m(V_i)} \\
		 &= \sum_{s=1}^M  z_s \tilde{c}_{s,i}^D +f^{D}_{i}.
	\end{aligned}
	\]

The discrete Gauss's law at the D-node $\bm{x}_j^D$ can be verified using the same approach. The effect of Step 2 on the discrete Gauss's law is then evaluated. After the update, the discrete Gauss's law at the D-node $\bm{x}_i^D$ is given by:

	\begin{align}
	2\kappa^2 \left(\operatorname{div}_D \mathcal{A}(\varepsilon^h) \tilde{\bm{E}}^h\right)_i^D &= 2\kappa^2 \left(\operatorname{div}_D \mathcal{A}(\varepsilon^h) \check{\bm{E}}^h\right)_i^D\\
	&\quad +\frac{2\kappa^2}{m(V_i)} \left[ \tau_{\{i<j\}}\varepsilon_{ij}^D  \frac{\eta}{\varepsilon_{ij}^D m(\partial V_{ij})}m(\partial V_{ij})+\tau_{\{i<l\}}\varepsilon_{il}^D \frac{-\eta}{\varepsilon_{il}^D m(\partial V_{il})}m(\partial V_{il})\right]\\
	&=\sum_{s=1}^M  z_s \tilde{c}_{s,i}^D +f^{D}_{i} +\frac{2\kappa^2}{m(V_i)}(\eta -\eta)= \sum_{s=1}^M  z_s \tilde{c}_{s,i}^D +f^{D}_{i}.
	\end{align}
Analogously, one has
	\begin{align}
	2\kappa^2 \left(\operatorname{div}_D \mathcal{A}(\varepsilon^h) \tilde{\bm{E}}^h\right)_j^D &= 2\kappa^2 \left(\operatorname{div}_D \mathcal{A}(\varepsilon^h) \check{\bm{E}}^h\right)_j^D\\
	&\quad +\frac{2\kappa^2}{m(V_j)} \left[ \tau_{\{j<i\}}\varepsilon_{ij}^D  \frac{\eta}{\varepsilon_{ij}^D m(\partial V_{ij})}m(\partial V_{ij})+\tau_{\{j<l\}}\varepsilon_{jl}^D \frac{\eta}{\varepsilon_{jl}^D m(\partial V_{jl})}m(\partial V_{jl})\right]\\
	&=\sum_{s=1}^M  z_s \tilde{c}_{s,j}^D +f^{D}_{j} +\frac{2\kappa^2}{m(V_j)}(-\eta +\eta)= \sum_{s=1}^M  z_s \tilde{c}_{s,j}^D +f^{D}_{j},
	\end{align}
and
	\begin{align}
	2\kappa^2 \left(\operatorname{div}_D \mathcal{A}(\varepsilon^h) \tilde{\bm{E}}^h\right)_l^D &= 2\kappa^2 \left(\operatorname{div}_D \mathcal{A}(\varepsilon^h) \check{\bm{E}}^h\right)_l^D\\
	&\quad +\frac{2\kappa^2}{m(V_l)} \left[ \tau_{\{l<i\}}\varepsilon_{il}^D  \frac{-\eta}{\varepsilon_{il}^D m(\partial V_{il})}m(\partial V_{il})+\tau_{\{l<j\}}\varepsilon_{jl}^D \frac{\eta}{\varepsilon_{jl}^D m(\partial V_{jl})}m(\partial V_{jl})\right]\\
	&=\sum_{s=1}^M  z_s \tilde{c}_{s,l}^D +f^{D}_{l} +\frac{2\kappa^2}{m(V_l)}(\eta -\eta)= \sum_{s=1}^M  z_s \tilde{c}_{s,l}^D +f^{D}_{l}.
	\end{align}
Hence, the discrete Gauss's law is rigorously satisfied in two successive local relaxation steps.

\end{proof}

\begin{thm}\label{thm-positive}
There exists a unique minimizer $\delta_{ c_s}^*$ to function $\Delta \mathcal{F}_1$~\reff{disfunc1} in the interval $B:=(I_l, I_r)$, where $$I_l=\max \left\{ -c^{D}_{s,j} m(V_j), -m(V_i)\frac{1-\sum_{v=1}^M a_v^3 c_{v,i}^D}{a_s^3}\right\} \text{ and } I_r=\min \left\{ c^{D}_{s,i} m(V_i), m(V_j)\frac{1-\sum_{v=1}^M a_v^3 c_{v,j}^D}{a_s^3}\right\}$$ are chosen to guarantee the positivity of both the updated concentrations of ions~\eqref{deltac} and solvent molecules~\eqref{deltac0}.
\end{thm}

\begin{proof}
	One first establishes the convexity of the energy change $\Delta \mathcal{F}_1$~\eqref{disfunc1}. Taking the second derivative of $\Delta \mathcal{F}_1$ with respect to $\delta_{ c_s}$, one has
	$$
	\begin{aligned}
		\frac{\partial^2 \Delta \mathcal{F}_1}{\partial \delta_{ c_s}^2}
		=&\frac{1}{2\kappa^2} \varepsilon_{ij}^D d_{km}^V m(\partial D_{km})\left( \frac{z_s }{\varepsilon_{ij}^D m(\partial V_{ij})} \right)^2
		+ \frac{1}{m(V_i)c_{s,i}^D-\delta_{ c_s}}  + \frac{1}{m(V_j)c_{s,j}^D+\delta_{ c_s}}\\
		&+\frac{a_s^6}{a_0^3}\frac{1}{m(V_i)\left(1-\sum_{v=1}^M a_v^3 c_{v,i}^D\right)+a_s^3 \delta_{ c_s}}+\frac{a_s^6}{a_0^3}\frac{1}{m(V_j)\left(1-\sum_{v=1}^M a_v^3 c_{v,j}^D\right)-a_s^3 \delta_{ c_s}}.
	\end{aligned}
	$$
Clearly, it is positive over the interval $B$. Therefore, the energy change~\eqref{disfunc1} is a strictly convex function over $B$.
 Let us consider a closed domain $B_{\gamma}=[I_l+\gamma, I_r-\gamma]\subset B$, where $0<\gamma<(I_r-I_l)/2$. It is clear that for any $\gamma$, there exists a minimizer of $\Delta \mathcal{F}_1$ within $B_{\gamma}$. We now aim to prove that when $\gamma$ is sufficiently small, the minimizer cannot be located at one of the boundary points of $B_{\gamma}$.
	
	Assume that the minimizer were achieved at $I_l+\gamma$. If $I_l=-c_{s,j}^Dm(V_j)$ and the minimizer $\delta_{ c_s}^*=-c^{D}_{s,j} m(V_j)+\gamma$, then
	$$
	\begin{aligned}
		\left.\frac{\partial \Delta \mathcal{F}_1}{\partial \delta_{ c_s}}\right|_{\delta_{ c_s}^*}=&- d_{km}^V m(\partial D_{km})\frac{z_s}{ m(\partial V_{ij})}\left(E_{ij}^D -\frac{z_s \delta_{ c_s}^*}{2\kappa^2\varepsilon_{ij}^D m(\partial V_{ij})} \right) -\log  \frac{c_{s,i}^D-\delta_{ c_s}^*/m(V_i)}{\gamma/m(V_j)}\\
		&+\frac{a_s^3}{a_0^3}\log \frac{1-\sum_{v=1}^M a_v^3 c_{v,i}^D+a_s^3 \delta_{ c_s}^*/m(V_i)}{1-\sum_{v=1}^M a_v^3 c_{v,j}^D-a_s^3 \delta_{ c_s}^*/m(V_j)}-\frac{\chi z_s^2}{a_s} \left(\frac{1}{\varepsilon_i^D}-\frac{1}{\varepsilon_j^D}\right).
	\end{aligned}
	$$
	It is easy to verify that $\left.\partial \Delta F_1/\partial \delta_{ c_s}\right|_{\delta_{ c_s}^{*}}<0$ when $\gamma$ is sufficiently small. This contradicts the assumption that $\delta_{ c_s}^*$ is a minimizer. Otherwise, if $I_l=-m(V_i)\left(1-\sum_{v=1}^M a_v^3 c_{v,i}^D\right)/ a_s^3$ and $\delta_{ c_s}^{*}=-m(V_i)\left(1-\sum_{v=1}^M a_v^3 c_{v,i}^D\right)/ a_s^3+\gamma$, then
	$$
	\begin{aligned}
		\left.\frac{\partial \Delta \mathcal{F}_1}{\partial \delta_{ c_s}}\right|_{\delta_{ c_s}^{*}}=&- d_{km}^V m(\partial D_{km})\frac{z_s}{ m(\partial V_{ij})}\left(E_{ij}^D -\frac{z_s \delta_{ c_s}^{*}}{2\kappa^2\varepsilon_{ij}^D m(\partial V_{ij})} \right) -\log  \frac{c_{s,i}^D-\delta_{ c_s}^{*}/m(V_i)}{c_{s,j}^D+\delta_{ c_s}^{*}/m(V_j)}\\
		&+\frac{a_s^3}{a_0^3}\log \frac{a_s^3 \gamma/m(V_i)}{1-\sum_{v=1}^M a_v^3 c_{v,j}^D-a_s^3 \delta_{ c_s}^{*}/m(V_j)}-\frac{\chi z_s^2}{a_s} \left(\frac{1}{\varepsilon_i^D}-\frac{1}{\varepsilon_j^D}\right).
	\end{aligned}
	$$

It can be shown that $\left.\partial \Delta F_1/\partial \delta_{ c_s}\right|_{\delta_{ c_s}^{*}}<0$ when $\gamma$ is sufficiently small. This leads to a contradiction. Similarly, it can be demonstrated that the minimizer of Eq.~\eqref{disfunc1} cannot be achieved at the right boundary point $I_r-\gamma$ when $\gamma$ is sufficiently small. Therefore, there exists a unique minimizer $\delta_{ c_s}^*$ of the function $\Delta \mathcal{F}_1$ in $B$. The uniqueness can be established by the strict convexity of $\Delta \mathcal{F}_1$. With the chosen values of $I_l$ and $I_r$, it is straightforward to verify that the updated concentrations of ions~\eqref{deltac} and solvent molecules~\eqref{deltac0} are guaranteed to be positive.

\end{proof}

\begin{thm}\label{thm-energy}
	The discrete energy~\eqref{disf} decays monotonically in successive relaxation steps of the proposed algorithm.
\end{thm}

\begin{proof}
One first checks the Step 1. Since the optimal mass $\delta_{ c_s}^*$ is chosen by minimizing the energy change $\Delta \mathcal{F}_1$~\eqref{disfunc1}, one has
	$$
	\left. \Delta \mathcal{F}_1 \right|_{\delta_{ c_s}=\delta_{ c_s}^*}\leq  \left. \Delta \mathcal{F}_1 \right|_{\delta_{ c_s}=0}=0
	$$
	by the convexity by Theorem~\ref{thm-positive}.
In Step 2, the energy change after update reads
	\[
	\Delta \mathcal{F}_2 =\kappa^2 \frac{-\left(A_1 E_{ij}^D +A_2 E_{jl}^D -A_3 E_{il}^D\right)^2}{A_1/(\varepsilon_{ij}^D m(\partial V_{ij}))+A_2/(\varepsilon_{jl}^D m(\partial V_{jl}))+A_3/(\varepsilon_{il}^D m(\partial V_{il}))},
	\]
	which is clearly less than or equal to zero.
As a result, the energy decays monotonically in successive relaxation steps of the proposed algorithm.
\end{proof}
\begin{rem}
Since $c_{s,i}^D \left(\log c_{s,i}^D + \theta \right) \geq - e^{-(\theta +1)}$ for a constant $\theta>0$, the discrete energy \reff{disf} is bounded below. It follows from Theorem~\ref{thm-energy} that the proposed algorithm is convergent.
\end{rem}

\section{Numerical Results}


\subsection{Accuracy Performance}

We perform a series of numerical tests to validate the effectiveness and accuracy of the proposed algorithm. In order to simplify the analysis, we initially neglect the effects of ionic size and Born solvation energy. For this simplified case, we construct an exact solution for binary monovalent ionic concentrations.


$$c_s=e^{-z_s\cos(\pi x) \sin (\pi y)}, (x,y)\in[0,2]^2, s=1,2,$$
and determines the distribution of fixed charges
$$
f=2\pi^2 \cos(\pi x) \sin(\pi y) -e^{-\cos(\pi x) \sin (\pi y)}+e^{\cos(\pi x) \sin (\pi y)}.
$$
We conduct numerical simulations with $\epsilon_{\text{tol}}=10^{-10}$, $z_1=+1$, and $z_2=-1$, and compute the $L^{2}$ and $L^{\infty}$ numerical errors against the exact solution by
\[
E_{2}=\sqrt{ \frac{\sum_{i=1}^{M_D} \left(c_s\left(\bm{x}_{s,i}^D\right)-c_{s,i}^D\right)^2}{M_D}},\quad E_{\infty}= {\max_{i=1,\cdots,M_D} \left(c_s\left(\bm{x}_{s,i}^D\right)-c_{s,i}^D\right)},\quad s=1,2.
\]
Numerical errors on different meshes are presented in Tables~\ref{table-l2} and~\ref{table-linf}, where $h_{\max}:=\max_{i=1,\cdots,M_D} d_{ij}^D$. It can be observed that both the $L^2$ and $L^\infty$ errors decrease significantly with mesh refinement, and the numerical convergence orders are approximately $2$, as expected.


\begin{table}[H]
	\centering
	\setlength{\abovecaptionskip}{10pt}%
	\setlength{\belowcaptionskip}{0pt}%
	\begin{tabular}{ccccc}	
		\hline \hline
		$h_{\max}$  & $c_1$  & Order  & $c_2$    & Order  \\
		\hline
		0.06071   & 1.4674E-3 & -  & 1.4671E-3 & -      \\
		0.04917  & 1.0542E-3 & 1.5678 & 1.0541E-3 & 1.5672   \\
		0.03118  & 3.7572E-4 & 2.2650  & 3.7706E-4 & 2.2570      \\
		0.02483   & 2.3573E-4 & 2.0458 & 2.3712E-4 & 2.0356\\
		0.01414    & 6.9601E-5 & 2.1679 & 6.4498E-5 & 2.3136\\
		\hline	\hline
	\end{tabular}
	\caption{$L ^ 2$ error and convergence order for numerical solutions of $c_1$ and $c_2$.  }
	\label{table-l2}	
\end{table}

\begin{table}[t!]
	\centering
	\setlength{\abovecaptionskip}{10pt}%
	\setlength{\belowcaptionskip}{0pt}%
	\begin{tabular}{ccccc}	
		\hline \hline
		$h_{\max}$  & $c_1$  & Order  & $c_2$    & Order  \\
		\hline
		0.06071   & 5.0737E-3 & -  & 5.0495E-3 & -      \\
		0.04917   & 3.6383E-3 & 1.5764 & 3.6250E-3 & 1.5711   \\
		0.03118  & 1.4164E-3 & 2.2613  & 14411E-3 & 2.2400    \\
		0.02483    & 8.1461E-4 & 2.0476 & 8.2041E-4 & 2.0428\\
		0.01414    & 2.7218E-4 &  1.9481 & 2.2989E-4 & 2.2608\\
		\hline	\hline
	\end{tabular}
	\caption{$L ^ {\infty}$ error and convergence order for  numerical solutions of $c_1$ and $c_2$. }
	\label{table-linf}	
\end{table}

\subsection{Structure-preserving Properties} \label{num:prop}
\begin{figure}[t!]
	\centering
	\includegraphics[scale=0.45]{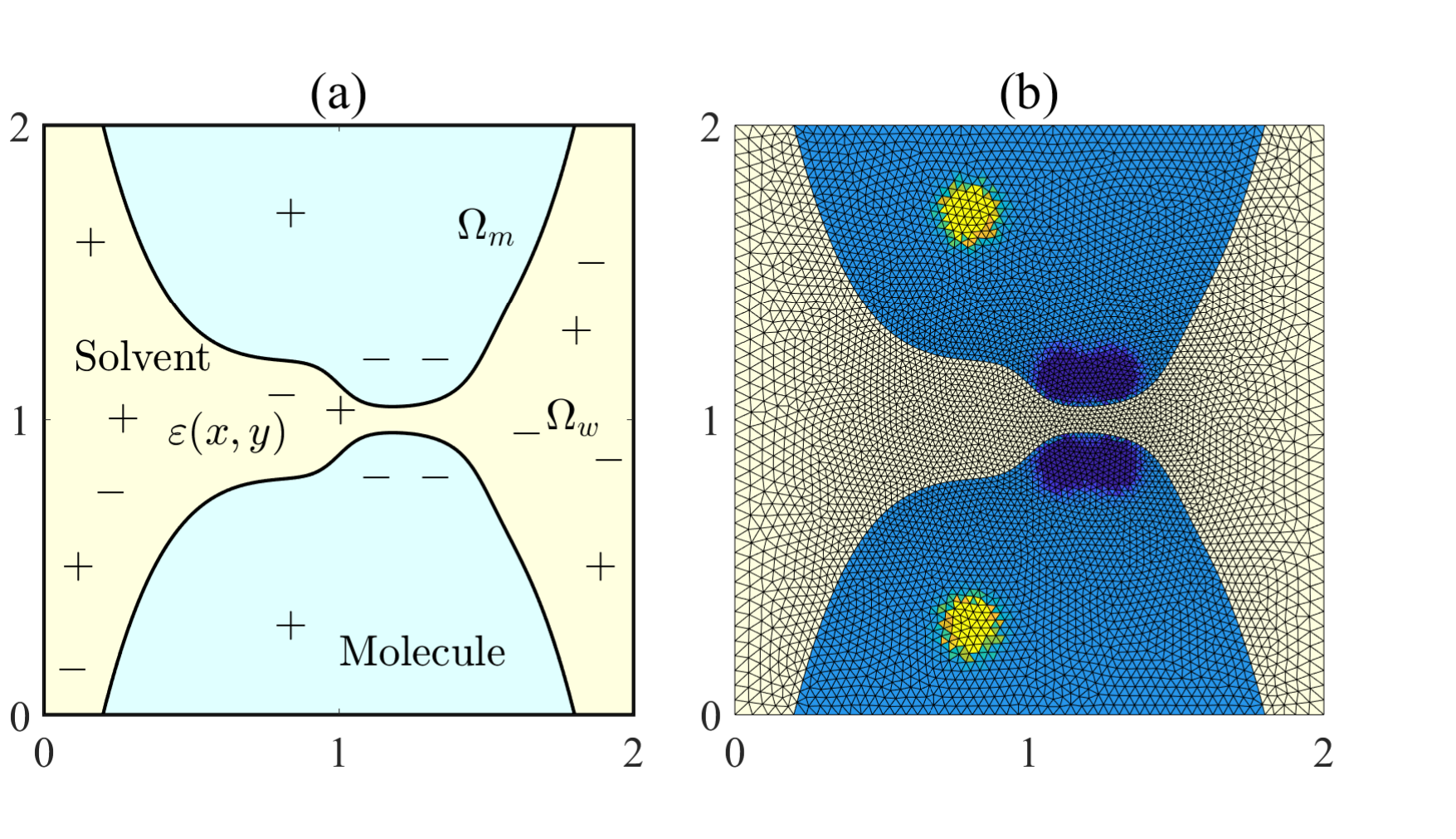} \vspace{-4mm}
	\caption{Illustration of the computational domain $\Omega$. (a) The molecule region contains fixed charges indicated by  plus and minus signs, and the solvent region has mobile ions indicated by plus and minus signs as well; (b) Delaunay triangulation of the molecule and solvent regions with highlighted resolution of fixed charges.}
	\label{f:geom}
\end{figure}

Numerical experiments are conducted to showcase the structure-preserving properties of the algorithm. We consider binary electrolyte solutions with charges $z_1=+1$ and $z_2=-1$, occupying a computational domain $\Omega=[0,2]^2\in \mathbb{R}^2$ under periodic boundary conditions. The domain is divided into a molecule region $\Omega_m$ and a solvent region $\Omega_w$, as depicted in Figure~\ref{f:geom} (a).
The space-dependent dielectric coefficient is set as
$$
\varepsilon(x,y)=\begin{cases}
	2,\quad (x,y)\in\Omega_m \\
	\frac{76}{1 + e^{22 (1 - 5|x-1.21|)}}+ 2,\quad (x,y)\in\Omega_w,
\end{cases}
$$
where the dielectric inhomogeneity mimics the dielectric response of the solvent along the narrow necking region.
The distribution of fixed charges in the molecule area $\Omega_m$ is given by
$$
\begin{aligned}
	f=&\omega\left\{-e^{-250\left[(x-1.12)^2+(y-0.85)^2\right]}-e^{-250\left[(x-1.12)^2+(y-1.15)^2\right]}\right.\\
	&-2e^{-250\left[(x-1.28)^2+(y-0.85)^2\right]}-e^{-250\left[(x-1.28)^2+(y-1.15)^2\right]}\\
	&\left.+e^{-250\left[(x-0.7)^2+(y-0.5)^2\right]}+e^{-250\left[(x-0.7)^2+(y-1.5)^2\right]}\right\},
\end{aligned}
$$
where $\omega$ regulates the strength of fixed charges. The spatial distribution of the fixed charges is illustrated in Figure~\ref{f:geom} (b).
Unless stated otherwise, the following parameters in numerical simulations are taken:
$a_1=1.10$, $a_2=1.05$, $a_0=1.25$, $\omega=2$, $\kappa=0.01$, and $ \chi=200$.
The initial ionic concentrations are uniform: $c_s(x,y)=0.1,(x,y)\in\Omega_w,s=1,2$.

\begin{figure}[t!]
	\centering
	\includegraphics[scale=0.31]{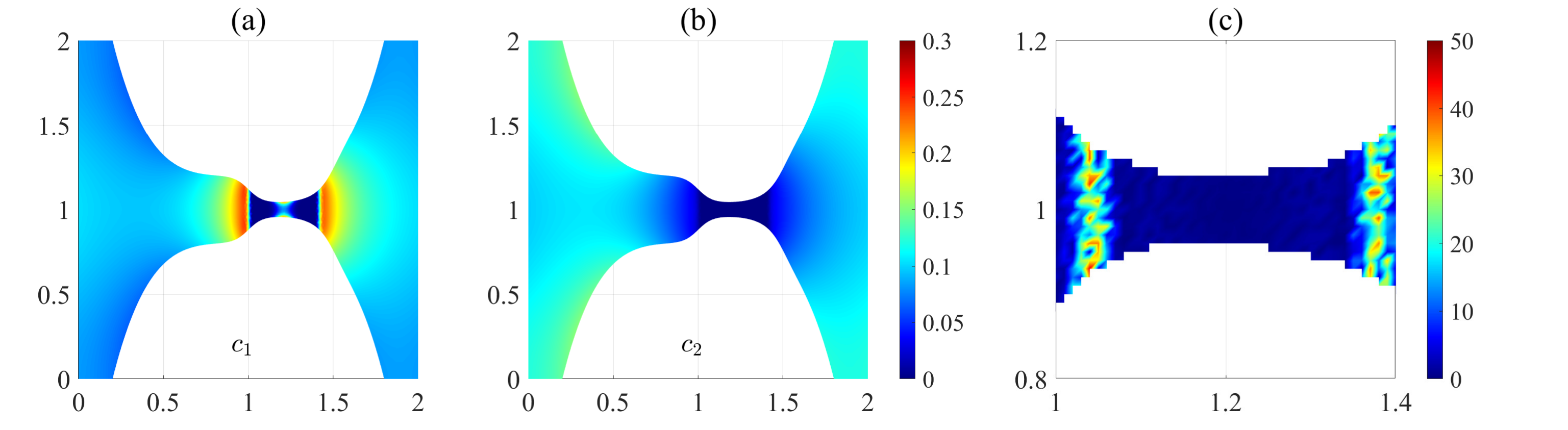} 
	\includegraphics[scale=0.31]{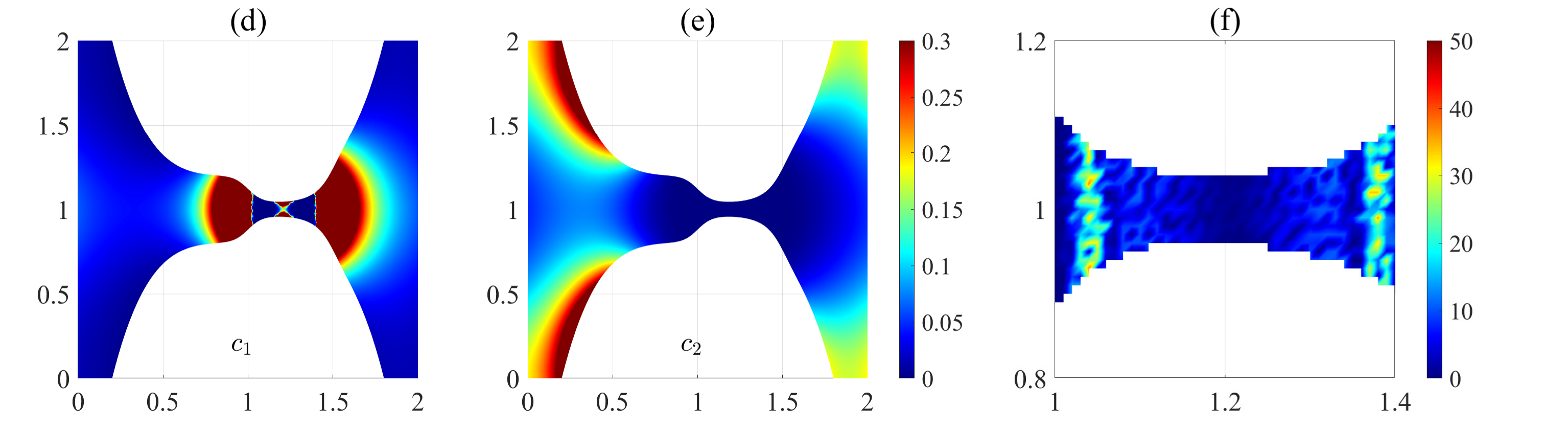}
	\vspace{-2mm}
	\caption{The concentration of cations (a,d), anions (b,e), and the distribution of $\text{P}_{1,ij}$ \reff{Pe} in the narrowest necking region (c,f) with $\kappa=0.01$ and various strength of fixed charges: (a-c) $\omega=2$ and (d-f) $\omega=10$.}
	\label{f:res}
\end{figure}

The proposed algorithm is utilized to solve the steady state of a heterogeneously charged system on a non-uniform mesh, as illustrated in Figure~\ref{f:geom} (b). It is important to note that the mesh is refined at the interface of two regions, the location of fixed charges, and narrow necking regions. Figure~\ref{f:res} displays the steady-state ionic distribution with $\kappa=0.01$. The plots reveal that both cations and anions are depleted from the narrow necking region due to the dominant repulsion caused by the sharp dielectric inhomogeneity.   Away from the necking region, the cation concentration peaks due to the electrostatic attraction exerted by the negative fixed charges in the molecule region.  With stronger fixed charges, as shown in the lower panel,  the concentration of counterions gets higher and even reaches saturation values. Such results clearly demonstrate that the steady-state ionic distribution is an outcome of the competition between electrostatic interactions and dielectric depletion.


We proceed to evaluate the influence of the rescaled Debye length, denoted as $\kappa$, on the concentrations of ions. A thorough comparison between Figure~\ref{f:res_k} and Figure~\ref{f:res} reveals that a smaller value of $\kappa$ leads to more pronounced boundary layers within the electric double layers (EDLs). The competition between electrostatic interactions and the dielectric effect becomes significantly intensified. For instance, in Figure~\ref{f:res_k} (a), the dielectric depletion region appears as two small regions, exhibiting a noticeable reduction compared to Figure~\ref{f:res} (a). Moreover, in Figure~\ref{f:res} (d), the dielectric depletion region completely disappears when $\kappa=0.001$ and $\omega=10$. Notably, wider saturation layers become evident in Figure~\ref{f:res_k} (e) as the electrostatic interactions are enhanced.



\begin{figure}[htpp]
	\centering
	\includegraphics[scale=0.27]{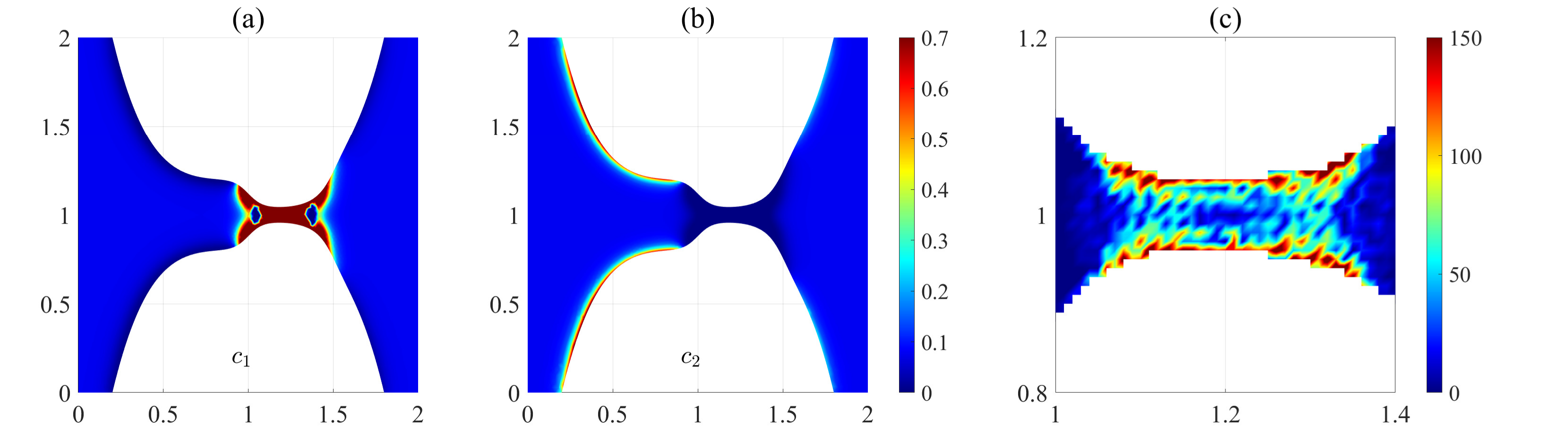}
	\includegraphics[scale=0.27]{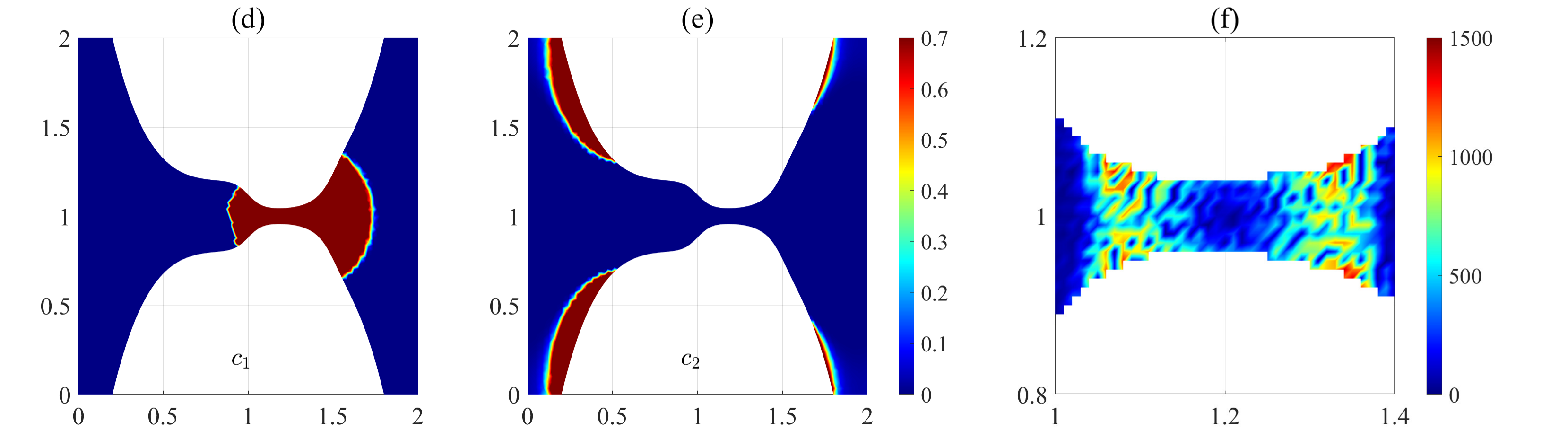}
	\vspace{-2mm}
	\caption{The concentration of cations (a,d), anions (b,e), and the distribution of $\text{P}_{1,ij}$ \reff{Pe} in the narrowest necking region (c,f) with $\kappa=0.001$ and various strength of fixed charges: (a-c) $\omega=2$ and (d-f) $\omega=10$.}
	\label{f:res_k}
\end{figure}

As observed previously, the presence of a small Debye length and strong fixed charges can result in the formation of sharp boundary layers within the EDLs, posing a significant numerical challenge in simulating charged systems~\cite{ZhangWangZhangLu_JCP22}. Additionally, another difficulty arises when the ionic concentration reaches saturation values~\cite{Horng2022Frontiers}.  We here assess the performance of the proposed local structure-preserving relaxation method in solving such challenging scenarios. To facilitate a quantitative comparison, we first introduce a dimensionless number to measure sharp changes across mesh nodes for each species:

\begin{equation}\label{Pe}
\text{P}_{s,ij}=\left| -z_s E_{ij}^D d_{ij} +d_{ij} \operatorname{grad}_D \left[-\frac{a_s^3}{a_0^3} \log \left(1-\sum_{v=1}^M a_v^3c_v^h\right) + \frac{\chi z_s^2}{a_s} \left(\frac{1}{\varepsilon^h}\right)\right]_{ij}^D \right|,
\end{equation}
which is a counterpart to the discrete P\'{e}clet number that quantifying the dominance of convection over diffusion in the time-dependent Nernst-Planck equations. Normally, the classical central differencing based numerical schemes start to give rise to severe spurious oscillations in the strong convection regime where $\text{P}_{s,ij}$ is greater than $2$.

The spatial distribution of $\text{P}_{1,ij}$ is depicted in the third column of Figures~\ref{f:res} and~\ref{f:res_k} for $\kappa=0.01$ and $\kappa=0.001$, respectively. It is evident that the proposed local structure-preserving relaxation method yields accurate ionic concentrations without any spurious oscillations, even when the system reaches ionic saturation and $\text{P}_{1,ij}$ reaches values as high as 1500 in the narrow necking region for $\kappa=0.001$ and $\omega=10$. This robust performance of the proposed method underscores its suitability for simulating charged systems.


Furthermore, we assess the ability of the proposed algorithm to preserve important structural properties, including the discrete Gauss's law, mass conservation of ions, and monotonic energy dissipation.

Define the residual errors
\[
\begin{aligned}
&R_{c_s}:=\left |\sum_{i=1}^{M_D} c_{s,i}^{D} m(V_i)-N_s \right |, ~s=1, 2,\\
&R_{E}:=\left\|\left\{2\kappa^2 \left(\operatorname{div}_D \mathcal{A}(\varepsilon^h) \bm{E}^h\right)_i^D-\sum_{s=1}^M  z_s c_{s,i}^D -f^{D}_{i}\right\}_{i=1}^{M_D}\right\|_\infty.
\end{aligned}
\]
Table~\ref{table-mass} presents the maximum values of $R_{c_1}$, $R_{c_2}$, and $R_{E}$ obtained during the relaxation process. It is evident that the total mass of both cations and anions, as well as the discrete Gauss's law, are satisfied with machine precision throughout the entire relaxation process. Additionally, Figure~\ref{f:prop} (a) illustrates the evolution of the normalized discrete energy $F_h$ for various combinations of $\kappa$ and $\omega$. Notably, the discrete energy monotonically decays during the relaxation process. Interestingly, we observe that as the Debye length $\kappa$ decreases, the relaxation process requires more iterations due to the development of sharper boundary layers. Moreover, Figure~\ref{f:prop} (b) demonstrates that the numerical concentrations of both ions and solvent can become lower than $10^{-10}$ while remaining positive throughout the relaxation process.



\begin{table}[t!]
	\centering
	\setlength{\belowcaptionskip}{0pt}%
	\begin{tabular}{cccc}	
		\hline \hline
		& $R_{c_1}$ &  $R_{c_2}$ &  $R_{E}$ \\
		\hline
		$\kappa=0.01,\omega=2$ & 9.1593E-16   & 8.8818E-16 & 2.2709E-18    \\
		$\kappa=0.01,\omega=10$ &9.1593E-16 & 1.2212E-15 & 8.6736E-18\\
		$\kappa=0.001,\omega=2$ & 9.9920E-16 &  8.8818E-16 & 5.9360E-18 \\
		$\kappa=0.001,\omega=10$ &8.0491E-16& 1.0825E-15 & 1.2523E-17\\
		\hline	\hline
	\end{tabular}
	\caption{The maximum $R_{c_1}$, $R_{c_2}$, and $R_{E}$ during the relaxation process.}
	\label{table-mass}	
\end{table}

\begin{figure}[b!]
	\centering
	\includegraphics[scale=0.42]{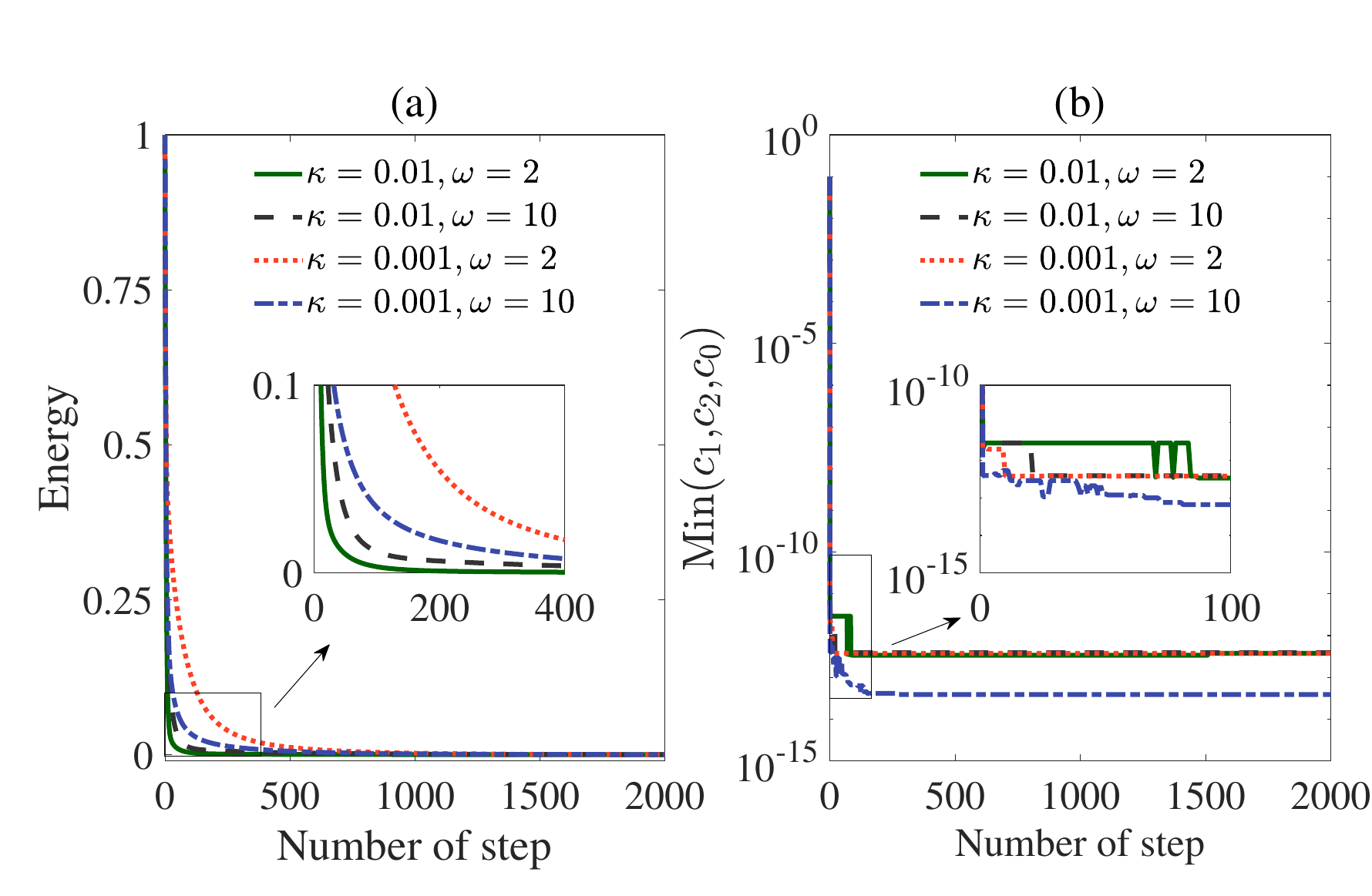} \vspace{-2mm}
	\caption{Evolution of the normalized discrete energy $F_h$ (a), and the minimum concentration of ions and solvent on Delaunay nodes (b) against relaxation steps with various combinations of $\kappa$ and $\omega$. }
	\label{f:prop}
\end{figure}

\section{Conclusions}\label{s:Con}

This paper presents a local relaxation method for tackling the numerical challenges associated with the steady-state modified Poisson--Nernst--Planck (PNP) equations in charged systems, considering ionic steric effects and the Born solvation energy. The presence of sharp boundary layers and ionic concentration saturation poses significant difficulties in simulations of these equations. By formulating the steady-state problem as a constraint optimization problem, we treat the ionic concentrations on unstructured Delaunay nodes as fractional particles moving along the edges between nodes. Consequently, the electric fields are updated to minimize the objective free energy while preserving the discrete Gauss's law. To ensure curl-free electric fields, we have developed a local relaxation method that inherently respects the discrete Gauss's law on unstructured meshes. Numerical analysis demonstrates that the optimal mass of the moving fractional particles ensures the positivity of both ionic and solvent concentrations. Furthermore, the free energy of the charged system monotonically decreases during the successive updates of ionic concentrations and electric fields. Further numerical results showcase the accuracy, numerical positivity, and dissipation of free energy achieved by the proposed method. Simulations of charged systems with small Debye length highlight the robustness of the relaxation method in handling sharp boundary layers and ionic saturation, underscoring its potential in molecular and plasma simulations.

\section*{Acknowledgements}

The last author expresses gratitude to Dr.\,Xiaozhe Hu for bringing the literature on mimetic finite difference to our attention.  This work is supported by the CAS AMSS-PolyU Joint Laboratory of Applied Mathematics. Z. Qiao's work is partially supported by the Hong Kong Research Grants Council (RFS Project No. RFS2021-5S03 and GRF project No. 15302122) and the Hong Kong Polytechnic University internal grant No. 1-9B7B. The work of Z. Xu and Q. Yin is partially supported by  NSFC (grant No. 12071288), Science and Technology Commission of Shanghai Municipality (grants No. 20JC1414100 and 21JC1403700), the Strategic Priority Research Program of CAS (grant No. XDA25010403) and the HPC center of Shanghai Jiao Tong University. S. Zhou's work is partially supported by the National Natural Science Foundation of China 12171319 and Science and Technology Commission of Shanghai Municipality (grant Nos. 21JC1403700).

\bibliographystyle{plain}
\bibliography{refbib}

\end{document}